\documentclass[10pt]{amsart}
\usepackage{amssymb}
\usepackage{latexsym}
\usepackage{mathrsfs}
\usepackage[pagebackref,colorlinks,linkcolor=red,citecolor=blue,urlcolor=blue,hypertexnames=false]{hyperref}
\usepackage{amsrefs}

\input{xy}
\xyoption{all}
\theoremstyle{plain}
\newtheorem{thm}{Theorem}[section]
\newtheorem{lem}[thm]{Lemma}
\newtheorem{coro}[thm]{Corollary}
\newtheorem{propo}[thm]{Proposition}
\newtheorem{defi}[thm]{Definition}
\newtheorem{rema}[thm]{Remark}

\theoremstyle{remark}

\newcommand{\cC}{\mathcal C}

\newcommand{\cK}{\mathcal K}

\newcommand{\cM}{\mathcal M}

\newcommand{\cO}{\mathcal O}

\newcommand{\Z}{\mathbb{Z}}
\newcommand{\Q}{\mathbb{Q}}
\newcommand{\C}{\mathbb{C}}
\newcommand{\M}{\mathscr{M}}
\newcommand{\hotimes}{\widehat{\otimes}}
\newcommand{\cotimes}{\underline\otimes}

\newcommand{\id}{{\rm id}}

\newcommand{\Ab}{{\rm Ab}}

\renewcommand{\top}{\rm top}
\renewcommand{\inf}{\rm inf}
\newcommand{\alg}{\rm alg}
\newcommand{\dif}{\rm dif}
\newcommand{\nil}{\rm nil}
\renewcommand{\top}{\rm top}

\newcommand{\diag}{\mathrm{diag}}

\newcommand{\iso}{\overset{\cong}{\longrightarrow}}
\def\Zz{\mathbb{Z}}
\def\fC{\mathfrak{C}}

\newcommand{\LC}{\mathfrak{LC}}
\begin{document}
\title{Algebraic $K$-theory and properly infinite $C^*$-algebras}

\author{Guillermo Corti\~{n}as}
\address{Guillermo Corti\~{n}as, Dep.\  Matem\'{a}tica-IMAS \\
 Ciudad Universitaria Pab 1\\
 (1428) Buenos Aires, Argentina.}
\email{gcorti@dm.uba.ar}
\urladdr{http://mate.dm.uba.ar/\~{}gcorti}
\thanks{Corti\~{n}as' research was supported by Conicet
 and partially supported by grants MTM2012-36917-C03-02,
 UBACyT 20020100100386,
 and PIP 11220110100800.}
\author{N.~Christopher Phillips}
\address{Department of Mathematics, University of Oregon,
      Eugene OR 97403-1222, USA,
      and Department of Mathematics, University of Toronto,
      Room 6290, 40 St.\  George St., Toronto ON M5S 2E4, Canada.}
\email{ncp@darkwing.uoregon.edu}
\urladdr{http://pages.uoregon.edu/ncp/}
\thanks{Phillips' research was partially supported by
  the US National Science Foundation under
  Grants DMS~0302401, DMS-0701076, and DMS-1101742.
  It was also partially supported by the Centre de Recerca
  Matem\`{a}tica (Barcelona) through a research visit conducted
  during 2011,
  and by the Research Institute for Mathematical Sciences
  of Kyoto University through a visiting professorship
  in 2011--2012.}
\date{12~Feb.\  2014}


\begin{abstract}
We show that several known results about the algebraic $K$-theory
of tensor products of algebras
with the $C^*$-algebra of compact operators in Hilbert space
remain valid for tensor products
with any properly infinite $C^*$-algebra.
\end{abstract}

\maketitle

\section{Introduction}

Let $\fC^*$ and $\Ab$ be the categories
of $C^*$-algebras and of abelian groups.
Let $\cO$ be a properly infinite $C^*$-algebra,
let $\cK = \cK (\ell^2 (\Z_{\ge 0}))$
be the $C^*$-algebra of compact operators,
let $(e_{j, k})_{j, k \in \Z_{\ge 0}}$
be the standard system of matrix units for~$\cK$,
and let $\cotimes$ be the spatial tensor product.
Consider the corner embedding $j \colon \cO \to \cO \cotimes \cK$,
given by $j (a) = a \cotimes e_{0, 0}$ for $a \in \cO$.
We show in Proposition~\ref{basic}
that if $E \colon \fC^* \to \Ab$ is an $M_2$-stable functor
then $E (j)$ is an isomorphism
\begin{equation}\label{intro:basic}
E (j) \colon E (\cO) \iso E (\cO \cotimes \cK).
\end{equation}

We use this to show that several known results
concerning the algebraic $K$-theory
of tensor products of algebras with $\cK$
remain valid for tensor products with $\cO$.
In these results,
and throughout the paper,
all topological vector spaces
(in particular, locally convex algebras) are assumed complete.
Also, $K_*$ denotes algebraic $K$-theory,
and $K_*^{\top}$ denotes a suitable version of topological $K$-theory
(the usual one when restricted to Banach algebras).
For example we prove in Theorem~\ref{karc*}
that if $A$ is a $C^*$-algebra then the comparison map is an isomorphism
\begin{equation}\label{intro:karc*}
K_* (A \cotimes \cO)\iso K_*^{\top} (A \cotimes \cO).
\end{equation}
In fact this is immediate from \eqref{intro:basic}
and the Karoubi-Suslin-Wodzicki theorem (\cite{karcomp}, \cite{sw}),
according to which
$K_* (A \cotimes \cK) \to K^{\top}_* (A \cotimes \cK)$ is an isomorphism.
We also prove that if $L$ is a
locally multiplicatively convex Fr\'{e}chet algebra
with a uniformly bounded one-sided approximate identity
and $\hotimes$ is the projective tensor product,
then the comparison map
\[
K_* (L \hotimes \cO) \to K_*^{\top} (L \hotimes \cO)
\]
is an isomorphism.
The analogous result for $L \hotimes \cK$ is \cite{ctc}*{Theorem 8.3.3}
(see also \cite{friendly}*{Theorem 12.1.1}).

We prove in Theorem~\ref{agree}
that if $L$ is a locally convex algebra then the various possible definitions of topological $K$-theory for locally convex algebras
all agree on $L \hotimes \cK$,
and moreover they coincide with Weibel's homotopy algebraic $K$-theory:
\begin{equation}\label{intro:khtop}
KH_* (L \hotimes \cO)\cong K^{\top}_* (L \hotimes \cO).
\end{equation}
The analog of \eqref{intro:khtop} for $L \hotimes \cK$ follows from \cite{ctc}*{Theorem 6.2.1}.
Further, let $HC_*$ be algebraic cyclic homology of $\Q$-algebras.
Then we show in Theorem \ref{seq} that there is a six term exact sequence
\[
\xymatrix{ K^{\top}_{1} (L \hotimes \cO)\ar[r]
 & HC_{2 n - 1} (L \hotimes \cO)\ar[r]
 & K_{2 n} (L \hotimes \cO)\ar[d]
\\
K_{2 n - 1} (L \hotimes \cO) \ar[u]
 & HC_{2 n - 2} (L \hotimes \cO) \ar[l]
 & K_0^{\top} (L \hotimes \cO). \ar[l]}
\]
The corresponding statement for $L \hotimes \cK$
is a particular case of \cite{ctc}*{Theorem 6.3.1}.

Finally we show, in Theorem \ref{khseq},
that if $A$ is any $\C$-algebra
and $\otimes_{\C}$ denotes the algebraic tensor product over~$\C,$
then
\[
KH_n (A \otimes_{\C} \cO) =
\begin{cases}
 K_0 (A \otimes_{\C} \cO) & {\mbox{$n$ even}}
 \\
 K_{-1} (A \otimes_{\C} \cO) & {\mbox{$n$ odd}},
\end{cases}
\]
and there is a six term exact sequence
\[
\xymatrix{
K_{-1} (A \otimes_{\C} \cO) \ar[r]
 & HC_{2 n - 1} (A \otimes_{\C} \cO)\ar[r]
 & K_{2 n} (A \otimes_{\C} \cO) \ar[d] \\
K_{2 n - 1} (A \otimes_{\C} \cO) \ar[u]
 & HC_{2 n - 2} (A \otimes_{\C} \cO) \ar[l]
 & K_{0} (A \otimes_{\C} \cO). \ar[l]
}
\]
The analogous statement for $A \otimes_{\C} \cK$
is a particular case of \cite{ctc}*{Theorem 7.1.1}.

\bigskip

The first version of this paper and all its results date back to 2007.
Although we have both lectured on these results since then,
we had not until now distributed the manuscript.
Motivated by the recent article \cite{sm},
where a particular case of~\eqref{intro:karc*} is proved,
we decided to make our paper publicly available.

We are grateful to George Elliott for useful discussions.

\section{Properly infinite algebras and stability}

If $A$ is any ring, we write
\[
\iota_A \colon A \to M_2 (A) 
\]
for the canonical inclusion into the upper left corner,
given by $\iota_A (a) = \diag (a, 0)$ for $a \in A$.
When $A$ is understood, we just write~$\iota$.

\begin{defi}\label{M2Stab}
We say that a functor $E \colon \fC^* \to \Ab$
from the category of $C^*$-algebras
to the category of abelian groups
is {\emph{$M_2$-stable}}
if $E (\iota)$ is an isomorphism for every $C^*$-algebra~$A$.
\end{defi}

Note that this definition makes sense for other categories of
rings and algebras,
such as the category $\LC$ of complete locally convex algebras.

\begin{propo}\label{basic}
Let $E \colon \fC^* \to \Ab$ be a functor.
Assume that $E$ is $M_2$-stable.
Let $\cO$ be a properly infinite $C^*$-algebra
and let $\cK = \cK (\ell^2 (\Z_{\ge 0}))$
be the $C^*$-algebra of compact operators.
Then $E$ maps the inclusion $j \colon \cO \to \cO \cotimes \cK$,
defined by
$j (a) = a \cotimes e_{0, 0}$ for $a \in \cO$,
to an isomorphism.
\end{propo}

\begin{proof}
Because $\cO$ is properly infinite,
it contains a sequence $(u_n)_{n \in \Z_{> 0}}$
of isometries with orthogonal range idempotents;
write $u = (u_1, u_2, \dots) \in \cO^{1 \times \infty}$
for the corresponding infinite row.
Represent an element $a \in \cO \cotimes \cK$ as an infinite matrix
$a = (a_{j, k})_{j, k \in \Z_{\ge 0}}$.
For such an element~$a$,
set
\[
a_{0, +} = (a_{0, k})_{k \in \Z_{> 0}},
\quad
a_{+, 0} = (a_{j, 0})_{j \in \Z_{> 0}},
\quad {\mbox{and}} \quad
a_{+, +} = (a_{j, k})_{j, k \in \Z_{> 0}}.
\]
Thus
\[
a = \left[ \begin{matrix}
 a_{0, 0} & a_{0, +}  \\
 a_{+, 0} & a_{+, +}
 \end{matrix} \right].
\]
Define a $*$-homomorphism $\phi \colon \cO \cotimes \cK \to M_2 (\cO)$
by
\[
\phi(a)
 = \left[ \begin{matrix}
 a_{0, 0} & a_{0,+}u^*  \\
 u a_{+,0}& u a_{+,+} u^*
 \end{matrix} \right]
\]
for $a \in \cO \cotimes \cK$.
By construction,
the diagram
\[
\xymatrix{
\cO \cotimes \cK\ar[r]^{\phi} & M_2 (\cO)
\\
  & \cO \ar[ul]_j \ar[u]^{\iota}
}
\]
commutes.
Since $E (\iota)$ is an isomorphism by hypothesis,
it follows that $E (j)$ is injective and that $E (\phi)$
is surjective; it remains to show that $E (j)$ is surjective,
or equivalently
that $E (\phi)$ is injective.

For any $C^*$-algebra~$A$,
we denote by $M (A)$ its multiplier algebra.
Consider
the partial isometry
\[
v = \left[ \begin{matrix}
        1    &  0     &  0    & \cdots \\
        0    &  u_1   &  u_2  & \cdots \\
        0    &  0     &  0    & \cdots \\
      \vdots & \vdots & \vdots & \ddots
  \end{matrix} \right] 
  = \left[ \begin{matrix} 1 & 0 \\ 0 & u \\ 0 & 0  \end{matrix} \right]
  \in M (\cO \cotimes \cK).
\]
Define a homomorphism
\[
\psi \colon \cO \cotimes \cK \to \cO \cotimes \cK
\]
by $\psi (a) = v a v^*$ for $a \in \cO$.
We have $E (\psi) = \id_{E (\cO \cotimes \cK)}$
(for example, by \cite{friendly}*{Proposition 2.2.6},
taking $B$ there to be $M (\cO \cotimes \cK)$).
Let $\kappa \colon M_2 (\cO) \to \cO \cotimes \cK$
by the $*$-homomorphism defined by
\[
\kappa \left( \left[ \begin{matrix}
 a_{0, 0} & a_{0, 1}  \\
 a_{1, 0} & a_{1, 1}
 \end{matrix} \right] \right)
 = a_{0, 0} \otimes e_{0, 0}
   + a_{0, 1} \otimes e_{0, 1}
   + a_{1, 0} \otimes e_{1, 0}
   + a_{1, 1} \otimes e_{1, 1}.
\]
Then the following diagram commutes:
\[
\xymatrix{
\cO \cotimes \cK \ar[dr]^{\psi} \ar[r]^{\phi}
 & M_2 (\cO) \ar[d]^{\kappa}
 \\
 & \cO \cotimes \cK.}
\]
It follows that $E (\phi)$ is injective; this concludes the proof.
\end{proof}

Recall that a functor of $C^*$-algebras is {\emph{homotopy invariant}}
if it sends two homotopic homomorphisms to the
same homomorphism.

\begin{coro}\label{hig_proper}
Let $E$ and $\cO$ be as in Proposition~\ref{basic},
and further
assume that $E$ is split exact.
Then $E ({-} \cotimes \cO)$ is homotopy invariant.
\end{coro}

\begin{proof}
The proposition shows that $E ({-} \cotimes \cO)$ is $\cK$-stable;
since it is also split exact by hypothesis,
it is homotopy
invariant by Higson's homotopy invariance theorem
\cite{hig}*{Theorem 3.2.2}.
(Note the misprint there:
Definition 2.1.8 should be Definition 2.1.10.)
\end{proof}

In the next corollary,
we use the notion of {\emph{diffotopy invariance}}
(sometimes called diffeotopy invariance)
for functors of locally convex algebras,
taken from the first part of \cite{cut}*{Definition 4.1.1};
note that $A [0, 1]$ there is defined to be the space of smooth
functions from $[0, 1]$ to~$A$.
The definition is essentially the same as for homotopy invariance,
replacing
continuous homotopies with $C^{\infty}$ homotopies
(diffotopies;
in some papers called diffeotopies).
Recall that $\LC$ is the category of complete locally convex algebras.
We require that the multiplication
be jointly continuous
(as in \cite{cut}*{Definition~2.1}),
but we do not require the existence of submultiplicative seminorms.

\begin{coro}\label{lc_hig}
Let $E \colon \LC \to \Ab$
be an $M_2$-stable split exact functor,
and let $\cO$ be a properly infinite $C^*$-algebra.
Then the functor $L\mapsto E (L \hotimes \cO)$ is diffotopy invariant.
\end{coro}

\begin{proof}
Fix $L \in \LC$ and consider the functor $F \colon \fC^* \to \Ab$
given by $F(A) = E (L \hotimes (A \cotimes \cO))$.
By Corollary \ref{hig_proper},
$F$ is homotopy invariant.
Hence $E$ sends the evaluation maps
$L \hotimes \cC ([0, 1], \, \cO) \to L \hotimes \cO$
at both endpoints to the same map.
It follows that the same is true of the
evaluation maps
$\cC^{\infty} ([0, 1], \, L \hotimes \cO) \to L \hotimes \cO$,
since the latter factor through the former.
\end{proof}

We will also need the following variant of Proposition \ref{basic}.

\begin{coro}\label{variant}
Let $E \colon \LC \to \Ab$ be an $M_2$-stable functor.
Let $k \colon \cO \to \cO \hotimes \cK$
be the usual inclusion,
given by
$a \mapsto a \hotimes e_{0, 0}$ for $a \in \cO$.
Then $E (k)$ is a naturally split monomorphism.
\end{coro}

\begin{proof}
The map $j$ of Proposition \ref{basic} factors through~$k$.
\end{proof}

\section{Applications}

In this section we apply the results of the previous one
to study the algebraic $K$-theory
of tensor products of different classes of algebras
with properly infinite $C^*$-algebras.
We essentially show that
the results of \cite{ctc}, \cite{karcomp}, \cite{sw}, and \cite{wodk}
remain valid if we stabilize with properly infinite $C^*$-algebras
instead of the compact operators.

The first application concerns the comparison
between Quillen's algebraic $K$-theory and the usual topological
$K$-theory of $C^*$-algebras.
It is the properly infinite variant of Karoubi's conjecture
for $C^*$-algebras,
proved in \cite{sw} for the $\cK$-stable case.

We need a lemma.
See \cite{friendly}*{Remark 2.1.13}
for why something needs to be done here.
We will use the reasoning of this lemma
for other categories later.

\begin{lem}\label{AlgKM2}
Algebraic $K$-theory of $C^*$-algebras is $M_2$-stable
in the sense of Definition~\ref{M2Stab}.
\end{lem}

\begin{proof}
Let $A$ be a $C^*$-algebra.
Let ${\widetilde{A}}$ be its unitization
as an algebra over $\Zz$.
We then get a commutative diagram with split exact rows:
\[
\xymatrix{
0 \ar[r]
 & A \ar[r] \ar[d]_{\iota_A}
 & {\widetilde{A}} \ar[r] \ar[d]^{\iota_{\widetilde{A}}}
 & \Zz \ar[r] \ar[d]^{\iota_{\Zz}}
 & 0
\\
0 \ar[r]
 & M_2 (A) \ar[r]
 & M_2 ({\widetilde{A}} ) \ar[r]
 & M_2 (\Zz) \ar[r]
 & 0.
}
\]
Apply $K_n$, getting:
\[
\xymatrix{
0 \ar[r]
 & K_n (A) \ar[r] \ar[d]_{(\iota_A)_*}
 & K_n ( {\widetilde{A}} ) \ar[r] \ar[d]^{(\iota_{\widetilde{A}})_*}
 & K_n (\Zz) \ar[r] \ar[d]^{\iota_{(\Zz})_*}
 & 0
\\
0 \ar[r]
 & K_n (M_2 (A)) \ar[r]
 & K_n ( M_2 ({\widetilde{A}} ) ) \ar[r]
 & K_n (M_2 (\Zz) ) \ar[r]
 & 0.
}
\]
Since $A$ and $M_2 (A)$ are $C^*$-algebras,
it follows from \cite{sw}*{Corollary 10.4}
that the rows are exact.
The maps $(\iota_{\widetilde{A}})_*$ and $\iota_{(\Zz})_*$
are isomorphisms because ${\widetilde{A}}$ and $\Zz$ are unital.
So $(\iota_A)_*$ is an isomorphism by the Five Lemma.
\end{proof}

\begin{thm}\label{karc*}
Let $A$ and $\cO$ be $C^*$-algebras,
with $\cO$ properly infinite.
Then the comparison map
$K_n (A \cotimes \cO) \to K_n^{\top} (A \cotimes \cO)$ from algebraic
to topological $K$-theory is an isomorphism for all $n \in \Z$.
\end{thm}

\begin{proof}
Apply the functors $K_*$ and $K_*^{\top}$
to the map $j \colon A \cotimes \cO \to A \cotimes \cO \cotimes \cK$.
Since the comparison
map is natural,
we obtain a commutative diagram:
\[
\xymatrix{
K_* (A \cotimes \cO) \ar[d] \ar[r]
 & K_*^{\top} (A \cotimes \cO)\ar[d]
 \\
K_* (A \cotimes \cO \cotimes \cK) \ar[r]
 & K_*^{\top} (A \cotimes \cO \cotimes \cK).}
\]
The right vertical map is well known to be an isomorphism.
By Proposition~\ref{basic},
applied to the functor $E = K_* (A\cotimes {-})$,
and Lemma~\ref{AlgKM2},
the left vertical map is an isomorphism.
By \cite{karcomp}*{Th\'{e}or\`{e}me on page 254}
and \cite{sw}*{Theorem 10.9},
the same is
true of the horizontal map at the bottom.
It follows that the top horizontal map is an isomorphism.
\end{proof}

Our second application is the properly infinite variant
of \cite{ctc}*{Theorem 6.2.1~(iii)}.
It concerns projective tensor products
of locally convex algebras with properly infinite $C^*$-algebras,
and establishes that for such products,
all variants of topological $K$-theory coincide with each other
and with Weibel's homotopy algebraic $K$-theory.
We recall that there are several topological $K$-theory groups
one can associate to a locally convex algebra $L$:
\begin{itemize}
\item
The Bott-periodic Cuntz groups $kk^{\top}_* (\C, L)$.
(See \cite{cubiva}*{Definition 4.2},
and see Remark~\ref{kktop} below on conflicting notation.)
\item
The diffotopy $K$-groups $KD_n (L)$ for $n \in \Z$ (\cite{friendly}*{\S8.2}).
\item
The diffotopy Karoubi-Villamayor style groups $KV^{\dif}_n (L)$
for $n \ge 1$
(\cite{friendly}*{\S8.1}).
\end{itemize}
On the other hand,
$L$, just like any other ring,
has associated several algebraic $K$-groups:
\begin{itemize}
\item
Quillen's $K$-groups $K_n (L)$ for $n \in \Z$
(defined by Bass for $n < 0$).
\item
Weibel's homotopy algebraic $K$-groups $KH_n (L)$
for $n \in \Z$.
(See \cite{kh} and \cite{friendly}*{\S5}.
The ring $\Sigma L$ is defined after \cite{friendly}*{Example 2.3.2}.)
\item
The Karoubi-Villamayor algebraic $K$-groups $KV_n (L)$
for $n \ge 1$.
(In \cite{kv}, use the discrete norm $p (0) = 0$ and $p (a) = 1$
for $a \neq 0.$
The group $KV_n (L)$ is called $K^{-n} (L)$ in~\cite{kv}.
Also see \cite{friendly}*{\S4}.)
\end{itemize}
%For a discussion of all these algebraic and topological
% QQQ

\begin{rema}\label{kktop}
The group we are calling here $kk^{\top}_* (A, B)$
is called $kk^{\alg}_* (A, B)$ in \cite{cubiva}*{Definition 4.2},
and also in~\cite{cut}.
(Compare the introduction to \cite{cut}*{\S6.1}
with \cite{cubiva}*{\S4}.)
It is called $kk^{\top}_* (A, B)$ in~\cite{ctc}.
It is not the same as the groups $kk_* (A, B)$
in \cite{cubiva}*{Definition 15.4}.
It is also not the same as the algebraic bivariant group
$kk_* (A, B)$ of~\cite{CrTh}.
\end{rema}

\begin{thm}\label{agree}
Let $L$ be a locally convex algebra
and let $\cO$ be a properly infinite $C^*$-algebra;
put $\cM = L \hotimes \cO$.
Then there are natural isomorphisms
\[
KH_n (\cM)
\cong KD_n (\cM)
\cong kk_n^{\top} (\C, \cM)
\]
for $n \in \Z$,
\[
KH_n (\cM)
\cong KV^{\dif}_n (\cM)
\cong KV_n (\cM)
\]
for $n \geq 1$,
and
\[
KH_n (\cM)
\cong K_n (\cM)
\]
for $n \leq 0.$
\end{thm}

\begin{proof}
All the functors appearing in the theorem are $M_2$-stable.
Moreover,
there are natural transformations as follows:
$KV_n \to KH_n$ for $n \ge 1$
(coming from (44) on page 127 of~\cite{friendly}),
$KV^{\dif}_n \to KD_n$ for $n \ge 1$
(coming from the formula before \cite{friendly}*{Proposition 8.2.1}),
and $K_n \to KH_n$ for $n \in \Z$
(see the beginning of \cite{kh}*{\S1};
that paper, in contrast to our convention,
takes $KV_n$ to be defined for all $n \in \Z$).
For any locally convex algebra~$L$,
we have algebras $\Omega L$ as in~(31) on page 127 of~\cite{friendly}
and $\Omega^{\dif} L$ as in~(64) on page 145 of~\cite{friendly},
and
there is an obvious map $\Omega L \to \Omega^{\dif} (L)$.
This map induces natural transformations
$KV_n \to KV^{\dif}_n$ for $n \geq 1$
and (by taking direct limits)
$KH_n \to KD_n$ for $n \in \Z$.

By \cite{ctc}*{Theorem 6.2.1~(iii)},
except for the map $K_n \to KH_n$,
all these natural maps,
when applied to $\cM \hotimes \cK$,
become isomorphisms for all
those $n$ for which they are defined.
By \cite{ctc}*{Lemma 3.2.1~(ii)},
\cite{ctc}*{Theorem 6.2.1~(ii)},
and \cite{kh}*{Proposition 1.5~(i)},
the map $K_n (\cM \hotimes \cK) \to KH_n (\cM \hotimes \cK)$
is an isomorphism for $n \leq 0$.
(Our theorem says nothing about this map for $n \geq 1$.)
Next use Corollary~\ref{variant}
(including naturality of the splitting),
and the fact that a retract of an isomorphism is an isomorphism,
to show that all these natural maps
(again, except $K_n \to KH_n$ for $n \geq 0$)
are also isomorphisms
when evaluated at $\cM$.
We have verified all the isomorphisms in the statement of the theorem
not involving $kk_n^{\top} (\C, \cM)$.

{}From \cite{cut}*{Theorem 6.2.1 and Definition 2.3.2}
(see Remark~\ref{kktop} for the discrepancy in notation),
we get a natural isomorphism
$kk^{\top}_0 (\C, \cM \hotimes \cK) \cong K_0 (\cM \hotimes \cK)$.
Appealing to Corollary~\ref{variant} in the same way as above,
we then get a
natural isomorphism
\[
kk^{\top}_0 (\C, \cM) \cong K_0 (\cM).
\]

The natural isomorphism
$KH_n (\cM) \cong KD_n (\cM)$ for $n \in \Z$,
which we already proved,
implies that the functor
$KH_* ({-} \hotimes \cO)$ is diffotopy invariant.
This functor is $M_2$-stable because $KH_*$ is.
It is shown in the proof of \cite{cut}*{Proposition 6.1.2}
that the inclusion $M_2 (\C) \to \cK_{\infty}$
into the smooth compact operators
induces a diffotopy equivalence
$\cK \hotimes M_2 (\C) \to \cK \hotimes \cK_{\infty}$.
Therefore
\[
KH_* ({-} \hotimes \cO \hotimes \cK) 
 \to KH_* ({-} \hotimes \cO \hotimes \cK \hotimes \cK_{\infty})
\]
is an isomorphism.
Applying Corollary~\ref{variant} in the same way as before,
we conclude that $KH_* ({-} \hotimes \cO)$ is $\cK_{\infty}$-stable.
Summing up,
$KH_* ({-} \hotimes \cO)$ satisfies excision
(\cite{kh}*{Theorem~2.1})
and is diffotopy invariant and
$\cK_{\infty}$-stable.

The functor $kk^{\top}_* (\C, {-} \hotimes \cO)$
also has these properties.
For a manifold~$\M,$
possibly with boundary,
and a locally convex algebra~$A,$
let $C^{\infty} (\M, A)$ be the algebra of all infinitely often
differentiable $A$-valued functions on~$\M,$
with the topology of uniform convergence of all derivatives
in all seminorms on~$A$.
Define locally convex algebras
\[
C_{\infty} A
 = \big\{ f \in C^{\infty} ([0, 1], \, A) \colon
    {\mbox{$f^{(n)} (0) = 0$ for $n = 1, 2, \ldots$
    and $f^{(n)} (1) = 0$ for $n = 0, 1, 2, \ldots$}} \big\},
\]
\[
S_{\infty} A
 = \big\{ f \in C_{\infty} A \colon f (0) = 0 \big\},
\quad {\mbox{and}} \quad
S_0 A = \big\{ f \in C^{\infty} (S^1, A) \colon f (1) = 0 \big\}.
\]

Now consider the exact sequence
\[
0 \longrightarrow S_{\infty} A
  \longrightarrow C_{\infty} A
  \longrightarrow A
  \longrightarrow 0.
\]
(This is the same sequence as in \cite{cudocu}*{\S2.1}.)
The algebra $C_{\infty} A$ is diffotopy equivalent to~$0,$
so
$KH_* (C_{\infty} A) = kk^{\top}_* (\C, A) = 0.$
Using excision, we thus get for every $n \in \Z$
a commutative diagram with exact rows:
\[
\xymatrix{
0 \ar[r]
 & KH_{n + 1} (L \hotimes \cO) \ar[r] \ar[d]
 & KH_n ( S_{\infty} (L \hotimes \cO) ) \ar[r] \ar[d]
 & 0
\\
0 \ar[r]
 & kk^{\top}_{n + 1} (\C, \, L \hotimes \cO) \ar[r]
 & kk^{\top}_n (\C, \, S_{\infty} (L \hotimes \cO) ) \ar[r]
 & 0.
}
\]
Since
$KH_{0} (M \hotimes \cO)) \to kk^{\top}_{0} (\C, \, M \hotimes \cO)$
is an isomorphism for all locally convex algebras~$M,$
and since
$S_{\infty} (M \hotimes \cO) \cong S_{\infty} (M) \hotimes \cO,$
an induction argument shows that the map
$KH_{n} (L \hotimes \cO)) \to kk^{\top}_{n} (\C, \, L \hotimes \cO)$
is an isomorphism for all $n \geq 0.$

Forming the projective tensor product of a locally convex algebra~$A$
with the second extension in \cite{cudocu}*{\S2.3},
we further get the exact sequence
\[
0 \longrightarrow \cK_{\infty} \hotimes A
  \longrightarrow {\mathcal{T}}_0 \hotimes A
  \longrightarrow S_0 A
  \longrightarrow 0.
\]
By \cite{cudocu}*{Theorem~6.4},
we have
$KH_* ({\mathcal{T}}_0 \hotimes A)
 = kk^{\top}_{*} (\C, \, {\mathcal{T}}_0 \hotimes A) = 0$
for every locally multiplicatively convex algebra~$A$,
that is, every locally convex algebra whose topology
is given by submultiplicative seminorms.
We claim that the same argument works even without
local multiplicative convexity.
Indeed,
since the theories are defined for locally convex algebras,
the proof only needs to be modified at the end by
tensoring with the identity map on a general locally convex algebra,
and checking that allowing general locally convex algebras
does not affect the proof of \cite{cudocu}*{Lemma~6.3}.
(It is shown in \cite{cubiva}*{Proposition~8.2}
that $kk^{\top}_{*} (\C, \, {\mathcal{T}}_0 \hotimes A) = 0$
without local multiplicative convexity,
but the analog of \cite{cudocu}*{Theorem~6.4}
is not present in \cite{cubiva}.)
It follows that we get a commutative diagram with exact rows:
\[
\xymatrix{
0 \ar[r]
 & KH_{n + 1} (S_0 ( L \hotimes \cO)) \ar[r] \ar[d]
 & KH_n ( \cK_{\infty} \hotimes L \hotimes \cO) \ar[r] \ar[d]
 & 0
\\
0 \ar[r]
 & kk^{\top}_{n + 1} (\C, \, S_0 ( L \hotimes \cO) ) \ar[r]
 & kk^{\top}_n (\C, \, \cK_{\infty} \hotimes L \hotimes \cO) \ar[r]
 & 0.
}
\]
An argument similar to the one above,
now also using $\cK_{\infty}$-stability,
shows that
$KH_{n} (L \hotimes \cO) \to kk^{\top}_{n} (\C, \, L \hotimes \cO)$
is an isomorphism for all $n \leq 0.$
\end{proof}

In view of the theorem above,
the topological $K$-groups of a locally convex algebra
of the form $L \hotimes \cO$ are
unambiguously defined.
To unify notation, we write
\[
K^{\top}_* (L \hotimes \cO)
  = kk_*^{\top} (\C, L \hotimes \cO)
  = KD_* (L \hotimes \cO)
  = KH_* (L \hotimes \cO).
\]
Note in particular that these topological $K$-groups are Bott periodic,
by \cite{cubiva}*{Theorem~8.3}.

A \emph{Fr\'{e}chet algebra} is a locally convex algebra
which is metrizable;
we call it \emph{locally multiplicatively convex}
if its topology is determined
by a sequence $(\rho_n)_{n \in \Z_{> 0}}$
of submultiplicative seminorms.
A {\emph{uniformly bounded left approximate identity}}
in $L$ is a net $(e_{\lambda})_{\lambda \in \Lambda}$
of elements of $L$
such that $e_{\lambda} a \to a$
for all $a \in L$
and such that
$\sup_{n \in \Z_{> 0}} \sup_{\lambda \in \Lambda} \rho_n (e_{\lambda})
  < \infty$.
A {\emph{uniformly bounded right approximate identity}}
is defined similarly.

\begin{thm}\label{karlc}
Let $L$ be a locally multiplicatively convex Fr\'{e}chet
algebra with uniformly bounded left or
right approximate identity.
Then there is a natural isomorphism
$K_* (L \hotimes \cO) \cong K_*^{\top} (L \hotimes \cO)$.
\end{thm}

\begin{proof}
By \cite{friendly}*{Theorem 6.6.6},
$K$-theory satisfies excision
in the category $\fC$
of locally multiplicatively convex Fr\'{e}chet algebras
with uniformly bounded left approximate identities.
It follows from the same reasoning as in the proof of Lemma~\ref{AlgKM2}
that $K$-theory is $M_2$-stable on $\fC$.
Moreover $\fC$ is closed under $\hotimes$,
and any $C^*$-algebra is in $\fC$.
By \cite{friendly}*{Theorem 12.1.1},
the map
$K_* (L \hotimes \cO \hotimes \cK)
 \to K_*^{\top} (L \hotimes \cO \hotimes \cK)$
is an isomorphism.
The proof is completed by using Corollary~\ref{variant}
(including naturality of the splitting),
and the fact that a retract of an isomorphism is an isomorphism.
\end{proof}

\begin{rema}
For $L$ as in Theorem \ref{karlc},
$K_* (L \hotimes \cO)$ agrees also with
the topological $K$-theory
of locally multiplicatively convex Fr\'{e}chet algebras
defined in \cite{frechet}.
To see this,
it suffices,
by the argument of \cite{friendly}*{Remark 12.1.4},
to check that
$K_0 (L \hotimes \cO)\cong K_0 (L \hotimes \cO \hotimes \cK_{\infty})$.
This follows from Theorem \ref{agree},
using $\cK_{\infty}$-stability of $kk^{\top}$.
\end{rema}

In Theorem~\ref{seq} we consider algebraic cyclic homology
of $\Q$-algebras,
as in \cite{friendly}*{\S11.1}.
If $A$ is an algebra,
we write
$HC_* (A) = H_* (C^{\lambda} (A / \Q))$
for the homology of Connes' complex
defined using the algebraic tensor product,
taken over $\Q$:
\[
C_n^{\lambda} (A / \Q) = (A^{\otimes_{\Q} (n + 1)})_{\Z / (n + 1) \Z}.
\]

We will need some properties of infinitesimal $K$-theory $K^{\inf}_*$
(as defined, for example, in \cite{friendly}*{\S11.2}).

\begin{thm}[\cite{kabi}*{Remark 4.2},
 \cite{friendly}*{Theorem 11.2.1~(ii)}]\label{T_4129_Kinf}
On $\Q$-algebras,
the functor $K^{\inf}_*$ satisfies excision and is $M_2$-stable.
\end{thm}

\begin{proof}
Once one has excision,
$M_2$-stability follows from the same argument as in
the proof of Lemma~\ref{AlgKM2}.
As observed in \cite{kabi}*{Remark 4.2} and cited in \cite{friendly}*{Theorem 11.2.1~(ii)}, excision
follows from \cite{kabi}*{Theorem 0.1};
since the proof is not made explicit in the above references,
we give a short argument below.

We will need birelative groups for $K_*$, $K^{\inf}_*$,
and $HN_*$
as in~\cite{kabi}.
These are defined for rings $A$ and~$B$,
an ideal $I$ in $A$,
and a homomorphism $f \colon A \to B$ such that
$f |_I$ is injective and $f (I)$ is an ideal in~$B$.
For $K_*$, see \cite{GW}*{\S0.1};
in the other two cases,
the definitions are analogous.
By construction,
there is a long exact sequence
\[
%\cdots \longrightarrow
K_{n + 1} (B : I)
  \longrightarrow K_n (A, B : I)
  \longrightarrow K_n (A : I)
  \longrightarrow K_n (B : I)
  \longrightarrow K_{n - 1} (A, B : I),
%   \longrightarrow \cdots.
\]
and similarly for $K^{\inf}_*$ and $HN_*$.
(See \cite{friendly}*{\S11.2}.)
Moreover, the fibration of \cite{friendly}*{\S11.2}
involving these gives rise to a long exact sequence
\[
% \cdots \longrightarrow
HN_{n + 1} (A, B : I)
  \longrightarrow K^{\inf}_n (A, B : I)
  \longrightarrow K_n (A, B : I)
  \longrightarrow HN_n (A, B : I)
 \longrightarrow K^{\inf}_{n - 1} (A, B : I).
%   \longrightarrow \cdots.
\]
When $A$ and $B$ are $\Q$-algebras,
\cite{kabi}*{Theorem 0.1}
states that the map $K_* (A, B : I) \to HN_* (A, B : I)$
is an isomorphism,
so that $K^{\inf}_* (A, B : I) = 0.$
Thus $K^{\inf}_* (A : I) \to K^{\inf}_* (B : I)$
is an isomorphism,
which is excision.
\end{proof}

We will also need the long exact sequence,
for a $\Q$-algebra~$A,$
in the top row of \cite{friendly}*{Diagram (86)}
(before Remark 11.3.4 there):
\begin{equation}\label{weseq}
\xymatrix{
KH_{n + 1} (A) \ar[r]
& K^{\nil}_n (A) \ar[r]
& K_n  (A) \ar[r]
& KH_n (A)  \ar[r]
& K^{\nil}_{n-1} (A). }
\end{equation}

\begin{thm}\label{seq}
Let $L$ be a locally convex algebra
and let $\cO$ be a properly infinite $C^*$-algebra;
put $\cM = L \hotimes \cO$.
For each $n \in \Z$,
there is a natural six term exact sequence of abelian groups as follows:
\[
\xymatrix{
 K^{\top}_{1} (\cM)\ar[r]
  & HC_{2 n - 1} (\cM)\ar[r]
  & K_{2 n} (\cM)\ar[d]
   \\
 K_{2 n - 1} (\cM) \ar[u]
  & HC_{2 n - 2} (\cM) \ar[l]
  & K_0^{\top} (\cM). \ar[l]}
\]
\end{thm}

\begin{proof}
Let $\nu_n \colon K^{\nil}_n (\cM) \to HC_{n-1} (\cM)$
be the a natural map in the second column
of \cite{friendly}*{Diagram (86)}
(before Remark 11.3.4 there).
Given the exact sequence~(\ref{weseq}) with $A = \cM$,
if we prove that $\nu_*$ is an isomorphism,
we obtain an exact sequence which,
when combined with Theorem~\ref{agree}
and Bott periodicity,
gives that of the theorem.
To prove that $\nu_n \colon K^{\nil}_n \to HC_{n - 1} (\cM)$
is an isomorphism,
it suffices,
in view of
\cite{ctc}*{Proposition 3.1.4},
to show that $\cM$ is $K^{\inf}$-regular.
This means that for every $p \ge 1$ and $n \in \Z$,
and with $\cM [t_1, t_2, \ldots, t_p]$ being the polynomial
ring in $p$~variables over~$\cM$,
the map
$\cM \to \cM [t_1, t_2, \ldots, t_p]$ induces an isomorphism
in $K^{\inf}_*$.
Theorem~\ref{T_4129_Kinf} implies that
$K^{\inf}_*$ satisfies excision and is $M_2$-stable.
By Corollary~\ref{lc_hig},
it follows that $K^{\inf}_* ({-} \hotimes \cO)$ is diffotopy invariant.
The argument of the proof of \cite{comparos}*{Theorem~20},
using $C^{\infty}$~functions on $[0, 1]$
instead of continuous functions,
now shows that $\cM$ is $K^{\inf}$-regular.
\end{proof}

\begin{thm}\label{khseq}
Let $A$ be a $\C$-algebra
and let $\cO$ be a properly infinite $C^*$-algebra.
Then there is a natural isomorphism
\[
KH_n (A \otimes_{\C} \cO)
\cong
\begin{cases}
 K_0 (A \otimes_{\C} \cO) & {\mbox{$n$ even}}
 \\
 K_{-1} (A \otimes_{\C} \cO) & {\mbox{$n$ odd}},
\end{cases}
\]
and there is an exact sequence
\begin{equation*}
\xymatrix{
 K_{-1} (A \otimes_{\C} \cO)\ar[r]
 & HC_{2 n - 1} (A \otimes_{\C} \cO)\ar[r]
 & K_{2 n} (A \otimes_{\C} \cO)\ar[d]
  \\
K_{2 n - 1} (A \otimes_{\C} \cO) \ar[u]
 & HC_{2 n - 2} (A \otimes_{\C} \cO) \ar[l]
 & K_{0} (A \otimes_{\C} \cO). \ar[l]}
\end{equation*}
\end{thm}

\begin{proof}
We claim that $A \otimes_{\C} \cO$ is $K^{\inf}$-regular.
To begin with,
Theorem~\ref{T_4129_Kinf}
% implies that $K^{\inf}_* (- \otimes_{\C} \cO)$
implies that $K^{\inf}_* (A \otimes_{\C} -)$
is $M_2$-stable.
Proposition~\ref{basic} therefore
shows that the map
$K^{\inf}_* (A \otimes_{\C} \cO)
 \to K^{\inf}_* \big( A \otimes_{\C} (\cO \cotimes \cK) \big)$
is an isomorphism.
This map factors through
$K^{\inf}_* (A \otimes_{\C} \cO \otimes_{\C} \cK).$
Thus
\begin{equation}\label{Eq_4211_KInf}
K^{\inf}_* (A \otimes_{\C} \cO)
 \to K^{\inf}_* (A \otimes_{\C} \cO \otimes_{\C} \cK)
\end{equation}
is injective and has a natural splitting.
The algebra $A \otimes_{\C} \cO \otimes_{\C} \cK$
is $K^{\inf}$-regular by \cite{ctc}*{Theorem 6.5.3~(i)},
that is, the map
\[
K^{\inf}_* (A \otimes_{\C} \cO \otimes_{\C} \cK)
 \to K^{\inf}_* \big( (A \otimes_{\C} \cO \otimes_{\C} \cK)
      [t_1, t_2, \ldots, t_p] \big)
\]
is an isomorphism for all $p \in \Z_{> 0}.$
Using the natural splitting of the map~(\ref{Eq_4211_KInf})
(for $A$ and also with $A [t_1, t_2, \ldots, t_n]$ in place of~$A$)
and the fact that a retract of an isomorphism is an isomorphism,
one deduces that $A \otimes_{\C} \cO$ is also $K^{\inf}$-regular.
This is the claim.

It now follows from
\cite{ctc}*{Lemma 3.2.1~(ii)} and \cite{kh}*{Proposition 1.5~(i)}
that
$K_n (A \otimes_{\C} \cO) \to KH_n (A \otimes_{\C} \cO)$
is an
isomorphism for $n \le 0$.
In view of the sequence \eqref{weseq},
it only remains to show that $KH_* (A \otimes_{\C} \cO)$
satisfies Bott periodicity in~$A$.
By Proposition~\ref{basic}, there is a natural isomorphism
$KH_* (A \otimes_{\C} \cO)
 \cong KH_* (A \otimes_{\C} (\cO \cotimes \cK))$.
So it suffices to show that
$KH_* (A \otimes_{\C} (\cO \cotimes \cK))$
satisfies Bott periodicity in~$A$.
The tensor product of operators gives an isomorphism
\[
\mu \colon \cK (\ell^2 (\Z_{\ge 0})) \otimes \cK (\ell^2 (\Z_{\ge 0}))
 \to \cK (\ell^2 (\Z_{\ge 0} \times \Z_{\ge 0})).
\]
It makes
$KH_* (\cK) \cong K_* (\cK) \cong K^{\top}_* (\cK)$
into a graded ring
isomorphic to the Laurent polynomials $\Z [t, t^{-1}]$,
in which $t$ has degree $2$.
Next,
$1_A \otimes_{\C} (1_{\cO} \cotimes \mu)$
makes $KH_* (A \otimes_{\C} \cO \cotimes \cK)$ into a graded
$KH_* (\cK) = \Z[t,t^{-1}]$-module.
This implies Bott periodicity
for $KH_* (A \otimes_{\C} \cO \cotimes \cK)$.
\end{proof}

\end{document}